\documentclass[10pt,a4paper]{amsart}
\usepackage{amsmath}
\usepackage{amsfonts}

\usepackage{tikz}
\usetikzlibrary{matrix,arrows,decorations.pathmorphing}

\input xy
\xyoption{all}

\usepackage{color}

\def\cqfd{
{\hfill
\kern 6pt\penalty 500
\raise -1pt\hbox{\vrule\vbox to 5pt{\hrule width 4pt
\vfill\hrule}\vrule}}
\break}

\def\Gal{\mathop{\rm Gal}\nolimits}

\def\Im{\mathop{\rm Im}\nolimits}

\long\def\eg#1{{\color{green}#1}}
\long\def\eg#1{}

\newtheorem{theorem}{Theorem}
\newtheorem{proposition}[theorem]{Proposition}
\newtheorem{lemma}[theorem]{Lemma}
\newtheorem{corollary}[theorem]{Corollary}

\theoremstyle{definition}

\theoremstyle{remark}

\newtheorem{remark}[theorem]{Remark}

\begin{document}

\date{\today}

\title[Prym varieties]{On the number of rational points on Prym varieties over finite fields}

\author{Yves Aubry}
\address{Institut de Math\'ematiques de Toulon, Universit\'e du Sud Toulon-Var and
Institut de Math\'ematiques de Luminy, France}
\email{yves.aubry@univ-tln.fr}

\author{Safia Haloui}
\address{Department of Mathematics, Technical University of Denmark, Lyngby, Denmark}
\email{s.haloui@mat.dtu.dk}



\begin{abstract}
We give upper and lower bounds for the number of rational  points on Prym varieties over finite fields. Moreover,
we determine the exact maximum and minimum number of rational points on Prym varieties of dimension 2.
\end{abstract}

\maketitle




\section{Introduction}

Prym varieties are abelian varieties which come from unramified double coverings of curves.

Let $\pi : Y\longrightarrow X$ be an unramified covering of degree 2 of smooth absolutely irreducible projective algebraic curves defined over ${\mathbb F}_q$ with $q=p^e$, where $p$ is an odd prime number.
Let $\sigma$ be the non-trivial involution of this covering and $\sigma^{\ast}$ the induced involution on the Jacobian $J_Y$ of $Y$.
The \emph{Prym variety} $P_{\pi}$ (we will often drop the subscribe $\pi$ when it is clear from the context) associated to $\pi$ is defined as
$$P_{\pi}=\Im(\sigma^{\ast}-id).$$
It is also the connected component of the kernel of $\pi_{\ast}:J_Y\longrightarrow J_X$ which contains the origin of $J_Y$.

It is an abelian subvariety of $J_Y$  isogenous to a direct factor of $J_X$ in $J_Y$. If $X$ has genus $g+1\geq 2$, then $Y$ has genus $2g+1$ by Riemann-Hurwitz formula, and  $P_{\pi}$ has dimension $g$.

\bigskip

 Prym varieties coming from unramified double coverings of genus $g+1$ curves provide a family of principally polarized $g$-dimensional abelian varieties.
 Let us denote by ${\mathcal A}_g$ the moduli space of principally polarized abelian varieties of dimension $g$, by ${\mathcal J}_g$ the Jacobian locus in ${\mathcal A}_g$,
 by ${\mathcal P}_g$ the subset of ${\mathcal A}_g$ corresponding to Prym varieties, and by $\overline{\mathcal P}_g$ its closure, then  $\overline{\mathcal P}_g$ is an irreducible subvariety of
 ${\mathcal A}_g$, of dimension $3g$ (for $g\geq 5$), containing ${\mathcal J}_g$; for $g\leq 5$ one has  $\overline{\mathcal P}_g={\mathcal A}_g$ (see \cite{Beau}).

\bigskip

We are interested in the maximum and minimum number of rational points on  Prym varieties over finite fields. In \cite{per}, Perret proved that if $X$ has genus $g+1$,  $\pi : Y\longrightarrow X$ is a double unramified covering over ${\mathbb F}_q$, and $N(X)$ and $N(Y)$ are the respective numbers of rational points on $X$ and $Y$, then the number of rational points $\#  P({\mathbb F}_q)$ on the associated Prym variety $P$ satisfies
\begin{equation}
\label{PerretSup}
\#  P({\mathbb F}_q)\leq \Bigl(q+1+\frac{N(Y)-N(X)}{g}\Bigr)^{g}
\end{equation}
and
\begin{equation}
\label{PerretInf}
\#  P({\mathbb F}_q)\geq \Bigl(\frac{\sqrt{q}+1}{\sqrt{q}-1}\Bigr)^{\frac{N(Y)-N(X)}{2\sqrt{q}}-2\delta}(q-1)^{g}
\end{equation}
where $\delta=0$ if $\frac{N(Y)-N(X)}{2\sqrt{q}}+g$ is an even integer, and $\delta=1$ otherwise (here, we have corrected the value of $\delta$ given in \cite{per}).

 \bigskip

The aim of this paper is to give some new upper and lower bounds on the number of rational points on Prym varieties over finite fields. In the next section, we recall some methods (from \cite{AHLacta}) to estimate  the number of rational points on an abelian variety, knowing its trace. We also explain how to derive the Perret's bounds (\ref{PerretSup}) and (\ref{PerretInf}) in this setting.

In Section \ref{TracePrym}, we study the trace of a Prym variety. The results obtained can be combined with the bounds from Section \ref{VAbTrace} to obtain new bounds on $\#  P({\mathbb F}_q)$ which require less information on $P$ than the Perret's bounds (namely, they do not require the knowledge of $N(Y)$, and the last one is also independent from $N(X)$).

The last section is devoted to the study of Prym surfaces. The main result is exact formulas for the maximum and the minimum number of points on Prym surfaces.


\section{Bounding the number of rational points on an abelian variety depending on its trace}\label{VAbTrace}

Let $A$ be an abelian variety of dimension $g$ defined over a finite field  ${\mathbb F}_q$. The \emph{Weil polynomial} $f_A(t)$ of $A$ is the characteristic polynomial of its Frobenius endomorphism. It is a monic polynomial with integer coefficients and the set of its roots (with multiplicity) consists of couples of conjugated complex numbers of modulus $\sqrt{q}$.

Let $\omega_{1}, \dots, \omega_{g}, \overline{\omega}_{1}, \dots, \overline{\omega}_{g}$ be the roots of $f_A(t)$. For $1 \leq i \leq g$, we set $x_i = -(\omega_{i} + \overline{\omega}_{i})$. We say that $A$ is of \emph{type} $[x_1,\dots ,x_g]$. The \emph{trace} of $A$ is defined to be the trace of its Frobenius endomorphism. We denote by $\tau (A)$ the opposite of the trace of $A$, more explicitly:
$$\tau(A) = - \sum_{i = 1}^{g} (\omega_{i} + \overline{\omega}_{i}) = \sum_{i=1}^{g} x_{i}.$$
This is an integer, and since $\vert{x_i}\vert \leq 2 \sqrt{q}$, $i=1,\dots ,g$, we have $\vert{\tau(A)}\vert \leq 2g \sqrt{q}$.

\bigskip

In the case where our abelian variety is the Jacobian $J_X$ of a smooth projective absolutely irreducible curve $X/\mathbb{F}_q$, its trace can be easily expressed in terms of the number $N(X)$ of rational points on $X$. Indeed, we have
\begin{equation}
\label{N(X)}
\tau (J_X)=N(X)-(q+1)
\end{equation}
(it follows from the fact that the numerator of the zeta function of $X$ is the reciprocal polynomial of the Weil polynomial $f_{J_C}(t)$).

\bigskip

Now let $P$ be a Prym variety and $\pi : Y\longrightarrow X$ the associated unramified double covering.
The map
${\pi}_*\times (\sigma^*-id):J_{Y}\longrightarrow J_{X}\times P$    has finite kernel and sends the ${\ell}^n$-torsion points of $J_{Y}$ on those of $J_{X}\times P$, for any prime number $\ell$ distinct from the characteristic of ${\mathbb F}_q$.
Then, tensorising the Tate modules  by ${\mathbb Q}_{\ell}$, we get an  isomorphism of ${\mathbb Q}_{\ell}$-vector spaces
$$T_{\ell}(J_{Y})\otimes_{{\mathbb Z}_{\ell}}{\mathbb Q}_{\ell}{\longrightarrow} T_{\ell}(J_{X}\times P)\otimes_{{\mathbb Z}_{\ell}}{\mathbb Q}_{\ell}=T_{\ell}(J_{X})\otimes_{{\mathbb Z}_{\ell}}{\mathbb Q}_{\ell}\times T_{\ell}(P)\otimes_{{\mathbb Z}_{\ell}}{\mathbb Q}_{\ell}$$
which commutes with the action of the Frobenius. Therefore, we have
$$f_{J_Y}(t)=f_{J_X}(t)f_P(t).$$
It follows that $$\tau (J_Y)=\tau (J_X)+\tau (P),$$ and using (\ref{N(X)}), we get
\begin{equation}
\label{N(P)}
\tau (P)=N(Y)-N(X).
\end{equation}


\vskip1cm

Let us come back to general abelian varieties. With the same notations as before, we can write
$$f_A(t)=\prod_{i=1}^g(t-\omega_i)(t-\overline{\omega}_i)=\prod_{i=1}^g(t^2+x_it+q).$$
It is wellknown that the number of rational points on $A$ is
\begin{equation}
\label{card}
\# {A({\mathbb F}_q)} =f_A(1) = \prod_{i=1}^{g}(q+1+x_i).
\end{equation}
Since $\vert{x_i}\vert \leq 2 \sqrt{q}$, one deduces from \eqref{card} the classical \emph{Weil bounds}
\begin{equation}
\label{WeilBounds}
(q+1-2 \sqrt{q})^g\leq \# {A({\mathbb F}_q)} \leq(q+1+2 \sqrt{q})^g.
\end{equation}

\vskip1cm

Now, for $\tau\in [-2g\sqrt{q};2g\sqrt{q}]$, define
\begin{equation}
\label{M(tau)}
M(\tau)=\left(q+1+\frac{\tau}{g}\right)^g
\end{equation}
and
\begin{equation}
\label{m(tau)}
m(\tau)=  (q+1+\tau -2(r(\tau)-s(\tau))\sqrt q)(q+1+2\sqrt q)^{r(\tau)}(q+1-2\sqrt q)^{s(\tau)}
\end{equation}
where  $r(\tau)  =  \left[\frac{g+\left[\frac{\tau}{2\sqrt q}\right]}{2}\right]$ and $s(\tau)  =  \left[\frac{g-1-\left[\frac{\tau}{2\sqrt q}\right]}{2}\right]$ (for a real number $x$, we denote by $[x]$ its integer part).

\bigskip

We have the following estimation of $\# {A({\mathbb F}_q)}$ (see \cite{AHLacta} and \cite{AHLcras}):
\begin{theorem}\label{MajMin}
If $A/\mathbb{F}_q$ is an abelian variety of dimension $g$, we have
$$m(\tau(A))\leq \# A({\mathbb F}_q) \leq M(\tau(A)).$$
\end{theorem}

Notice that in the case of Prym varieties, the upper bound of Theorem \ref{MajMin} together with (\ref{N(P)}) gives the Perret upper bound (\ref{PerretSup}). The lower bound (\ref{PerretInf}), comes from the fact (proved by Perret in the case of Prym varieties) that for any abelian variety,  we have
$$\#  A({\mathbb F}_q)\geq \Bigl(\frac{\sqrt{q}+1}{\sqrt{q}-1}\Bigr)^{\frac{\tau(A)}{2\sqrt{q}}-2\delta}(q-1)^{g}$$
where $\delta=0$ if $\frac{\tau(A)}{2\sqrt{q}}+g$ is an even integer and $\delta=1$ otherwise, which is always less precise than the lower bound from Theorem \ref{MajMin} (for more details, see \cite{AHLacta}).

\bigskip

In the next section, we will use Theorem \ref{MajMin} without knowing the value of $\tau(A)$ (but having an estimation). In order to do so, we need some basic results on the functions $M$ and $m$ defined by (\ref{M(tau)}) and (\ref{m(tau)}). These results are summarized in the following proposition:

\begin{proposition}\label{EtudeBornes}
The functions $M$ and $m$ are continuous and increasing on $[-2g\sqrt{q};2g\sqrt{q}]$.
\end{proposition}
\begin{proof}
The function $M$ is obviously continuous, and it is increasing because for $\tau\in[-2g\sqrt{q};2g\sqrt{q}]$, $q+1+\tau /g\geq q+1-2\sqrt{q}>0$.

Now, we focus on $m$. First, notice that the functions $r$ and $s$ are piecewise constant and therefore $m$ is piecewise an affine function with leading coefficient $(q+1+2\sqrt q)^{r(\tau)}(q+1-2\sqrt q)^{s(\tau)}>0$. Hence, the fact that $m$ is increasing follows from its continuity.

We now prove that $m$ is continuous. Let $k\in\{-g,\dots ,g-2\}$ be an integer which has the \emph{same parity} as $g$, and $\alpha\in [0;2[$. As $[\alpha ]\in\{0,1\}$ and $g+k$ and $g-k$ are non-negative even integers, we have
$$r(2\sqrt{q}(k+\alpha ))  =  \left[\frac{g+k+[\alpha ]}{2}\right]=\frac{g+k}{2}+\left[\frac{[\alpha ]}{2}\right]=\frac{g+k}{2}$$
and
$$s(2\sqrt{q}(k+\alpha ))  =  \left[\frac{g-1-k-[\alpha ]}{2}\right]=\frac{g-k}{2}+\left[\frac{-1-[\alpha ]}{2}\right]=\frac{g-k}{2}-1.$$
In particular, the functions $r$ and $s$ are constant on any interval of the form $[2k\sqrt{q};2(k+2)\sqrt{q}[$, where $k\in\{-g,\dots ,g-2\}$ has the same parity as $g$, and thus $m$ is continuous (in fact affine) on these intervals.

It remains to check that
$$\lim\limits_{\substack{\alpha \to 2 \\ \alpha <2}} m(2\sqrt{q}(k+\alpha ))=m(2\sqrt{q}(k+2 )).$$
The previous computations show us that
$$r(2\sqrt{q}(k+\alpha ))-s(2\sqrt{q}(k+\alpha ))=k+1,$$
and thus the first factor in the expression of $m$ is
$$q+1+2\sqrt{q}(k+\alpha )-2(r(2\sqrt{q}(k+\alpha ))-s(2\sqrt{q}(k+\alpha )))\sqrt{q}=q+1+2\sqrt{q}(\alpha -1).$$
We deduce that
$$m(2\sqrt{q}(k+\alpha ))=(q+1+2\sqrt{q}(\alpha -1))(q+1+2\sqrt q)^{\frac{g+k}{2}}(q+1-2\sqrt q)^{\frac{g-k}{2}-1}$$
and as
$$m(2\sqrt{q}(k+2 ))=(q+1-2\sqrt{q})(q+1+2\sqrt q)^{\frac{g+k}{2}+1}(q+1-2\sqrt q)^{\frac{g-k}{2}-2},$$
we have
$$m(2\sqrt{q}(k+\alpha ))=\frac{(q+1+2\sqrt{q}(\alpha -1))}{(q+1+2\sqrt q)}m(2\sqrt{q}(k+2 )),$$
and the result follows.
\end{proof}

Notice that we have $$m(-2g\sqrt{q})=(q+1-2 \sqrt{q})^g\;\mbox{ and }\; M(2g\sqrt{q})=(q+1+2 \sqrt{q})^g,$$
in particular, the bounds of Theorem \ref{MajMin} are always more precise than the Weil bounds (\ref{WeilBounds}) (but require more information on $A$).


\section{On the trace of Prym varieties}\label{TracePrym}

As before, let $A$ be an abelian variety defined over ${\mathbb F}_{q}$ of dimension $g$, $f_A(t)$ be its Weil polynomial, $\omega_{1}, \dots, \omega_{g}, \overline{\omega}_{1}, \dots, \overline{\omega}_{g}$ be the complex roots of $f_A(t)$, $x_i = -(\omega_{i} + \overline{\omega}_{i})$, $1 \leq i \leq g$, and
$$\tau (A)= - \sum_{i = 1}^{g} (\omega_{i} + \overline{\omega}_{i}) = \sum_{i=1}^{g} x_{i}$$
be the opposite of the trace of $A$.
For $k\geq 1$, we also define $\tau_k (A)$ to be the opposite of the trace of $A\times_{{\mathbb F}_{q}}{\mathbb F}_{q^k}$, that is,
$$\tau_k (A) = - \sum_{i = 1}^{g} (\omega_{i}^k + \overline{\omega}_{i}^k)$$
(hence we have $\tau_1 (A)=\tau (A)$).

\bigskip

We recall the following classical upper bound for $\tau_2 (A)$ (see \cite{Ihara}), which is a direct consequence of the Cauchy-Schwartz Inequality:
\begin{eqnarray}
\tau_2(A)=-\sum_{i=1}^{g} x_{i}^2+2gq\leq -\frac{1}{g}\Bigl(\sum_{i=1}^{g} x_{i}\Bigr)^2+2gq=\frac{-\tau (A)^2}{g}+2gq. \label{tau2}
\end{eqnarray}

\bigskip

Now let $P$ be a Prym variety and $\pi : Y\longrightarrow X$ the associated unramified double covering. We denote by $N_k(X)$ and $N_k(Y)$ the respective numbers of rational points on $X$ and $Y$ over ${\mathbb F}_{q^k}$, $k\geq 1$. The results from Section \ref{VAbTrace} tell us that
$$N_k(X)=q^k+1+\tau_k (J_X)$$
and
$$N_k(Y)=q^k+1+\tau_k (J_X)+\tau_k(P)=N_k(X)+\tau_k(P).$$

\begin{remark}
As $\pi$ is unramified and of degree $2$, the number of rational points on $Y$ must be even (it is twice the number of splitting rational points on $X$). Of course, this holds for any finite extension of the base field, and therefore, for $k\geq 1$, the $N_k(Y)$ are even, or in other words (recall that $q$ is supposed to be odd), we have
$$\tau_k(P)\equiv \tau_k (J_X) \mod 2 .$$
\end{remark}

\bigskip

Now, we give estimations of $\tau (P)$ which are independent from $Y$. We start by the following lemma:

\begin{lemma}\label{InequalityNi}
With the notations above, we have
$$0\leq N(Y)\leq 2N(X)\leq N_2(Y).$$
\end{lemma}
\begin{proof}
The first inequality is obvious. For the second one, use the fact that the image of a rational point is a rational point, and the number of points in the  preimage of a point is at most $2$. For the third one, if we denote by $B_d(Y)$ the number of points on $Y$ of degree $d$, we have:
$$N_2(Y)=B_1(Y)+2B_2(Y).$$
The set $X({\mathbb F}_q)$ can be partitioned into two subsets: the rational points which are splitting  and those which are inert in the covering $Y\longrightarrow X$. Denote respectively their cardinal by $s$ and $i$, we have $B_1(Y)\geq 2 s$ and $B_2(Y)\geq i$. Hence, $N_2(Y)\geq 2s+2i=2N(X)$.
\end{proof}

The two first inequalities of Lemma \ref{InequalityNi} give us immediately the following result, which is stated in \cite{per}:
\begin{proposition}[Perret]\label{InequalityTau1}
We have
$$\vert\tau (P)\vert\leq N(X).$$
\end{proposition}

Notice that the bound of Proposition \ref{InequalityTau1} is sharp when $X$ has few points (in particular, if $X$ has no rational points, then we get the exact value of $\tau$).

\bigskip

The third inequality of Lemma \ref{InequalityNi} gives us the following proposition:

\begin{proposition}\label{InequalityTau2}
We have
$$\vert\tau (P)\vert\leq \sqrt{g(q^2-1)-\frac{g(N(X)-q-1)^2}{g+1}-2g(N(X)-q-1)+4g^2q}.$$
\end{proposition}
\begin{proof} We have
\begin{eqnarray*}
2(q+1+\tau (J_X) )=2N_1^X & \leq & N_2^Y \\
& = & q^2+1+\tau_2 (J_X)+\tau_2(P)\\
& \leq & q^2+1-\tau (J_X)^2/(g+1)+2(g+1)q-\tau (P)^2/g+2gq,
\end{eqnarray*}
where the last inequality comes from (\ref{tau2}). Rearranging the terms, we find
$$\frac{\tau (P)^2}{g}\leq q^2-1-\frac{\tau (J_X)^2}{g+1}-2\tau (J_X)+4gq,$$
and using the fact that $\tau (J_X)=N(X)-q-1$, the result follows (notice that the second term in the previous inequality is necessarily non-negative).
\end{proof}

\begin{remark}
The third inequality of Lemma \ref{InequalityNi} is sharp when $X$ has many points. Indeed, we have
$$2N(X)\leq N_2(Y)\leq 2N_2(X)$$
(the last inequality is just the second inequality of Lemma \ref{InequalityNi} applied after a quadratic extension of the base field), and according to (\ref{tau2}), a curve with many points over ${\mathbb F}_q$ must have few points over ${\mathbb F}_{q^2}$.
\end{remark}

Now, recall that we have defined
$$M(\tau)=\left(q+1+\frac{\tau}{g}\right)^g$$
and
$$m(\tau)=  (q+1+\tau -2(r(\tau)-s(\tau))\sqrt q)(q+1+2\sqrt q)^{r(\tau)}(q+1-2\sqrt q)^{s(\tau)}$$
where  $r(\tau)  =  \left[\frac{g+\left[\frac{\tau}{2\sqrt q}\right]}{2}\right]$ and $s(\tau)  =  \left[\frac{g-1-\left[\frac{\tau}{2\sqrt q}\right]}{2}\right].$  Theorem \ref{MajMin} and Proposition \ref{EtudeBornes} give us:

\begin{corollary}\label{BornesN(X)}
We have
$$m(-N(X))\leq \# P({\mathbb F}_q) \leq M(N(X))$$
and
$$m\left(-\varphi(N(X))\right)\leq \# P({\mathbb F}_q) \leq M\left(\varphi(N(X))\right),$$
where
$$\varphi(N(X))=\sqrt{g(q^2-1)-\frac{g(N(X)-q-1)^2}{g+1}-2g(N(X)-q-1)+4g^2q}.$$
\end{corollary}

By combining Proposition \ref{InequalityTau1} and Proposition \ref{InequalityTau2}, we can eliminate the variable $N(X)$:
\begin{proposition}\label{InequalityTau3}
If $\vert\tau (P)\vert\geq q-g$ (for instance, this condition is satisfied when $g\geq q$), then we have
$$\vert\tau (P)\vert\leq \frac{g}{2g+1}\Bigl(q-g+\sqrt{(q-g)^2+(2g+1)(4gq+q^2+6q+1)}\Bigr).$$
\end{proposition}
\begin{proof}
The last inequality in the proof of Proposition \ref{InequalityTau2} can be rewritten as
$$\frac{\tau (J_X)^2}{g+1}+2\tau (J_X) +\frac{\tau (P)^2}{g}-q^2-4gq+1\leq 0.$$
Considering its first term as a polynomial equation in $\tau (J_X)$ and computing the roots, we find
$$\tau (J_X)\leq -(g+1)+\sqrt{(g+1)(q^2+g+4gq-\frac{\tau(P)^2}{g})}.$$
But Proposition \ref{InequalityTau1} tells us that $\tau (J_X)\geq\vert\tau (P)\vert -(q+1)$, and therefore, we have
$$\vert\tau (P)\vert +g-q\leq \sqrt{(g+1)(q^2+g+4gq-\frac{\tau(P)^2}{g})}.$$
Under the assumptions of the proposition, the first term of this inequality is non-negative, so we can raise everything to the square. We get
\begin{eqnarray*}
(\vert\tau(P)\vert +g-q)^2 & \leq & (g+1)(q^2+g+4gq-\frac{\tau(P)^2}{g}) \\
\tau(P)^2+2 (g-q)\vert\tau(P)\vert +g^2-2gq+q^2 & \leq & gq^2+g^2+4g^2q-\tau(P)^2+q^2+g+4gq-\frac{\tau(P)^2}{g}
\end{eqnarray*}
Hence
$$-(2g+1)\tau(P)^2-2g(g-q)\vert\tau(P)\vert +g^2(4gq+q^2+6q+1)\geq 0.$$

Considering the first term of this last inequality as a polynomial equation in $\vert\tau (P)\vert$ and computing the roots, we get the result.
\end{proof}

The bound of Proposition \ref{InequalityTau3} is sharper than the Weil bound $\vert\tau (P)\vert\leq 2g\sqrt{q}$ if
$$2g\sqrt{q} \geq \Bigl(q-g+\sqrt{(q-g)^2+(2g+1)(4gq+q^2+6q+1)}\Bigr)g/(2g+1),$$
and as the second member is the greatest root of the polynomial in $\vert\tau (P)\vert$ considered at the end of the proof of Proposition \ref{InequalityTau3}, and the smallest root must be smaller than $2g\sqrt{q}$, this inequality is equivalent to
\begin{eqnarray*}
0 & \leq & (2g+1)(2g\sqrt{q})^2+2g(g-q)2g\sqrt{q} -g^2(4gq+q^2+6q+1)\\
0 & \leq & (2g+1)4q+4(g-q)\sqrt{q} -4gq-q^2-6q-1\\
g & \geq & (q^2+4q\sqrt{q}+2q+1)/(4q+4\sqrt{q})=(q\sqrt{q}+3q-\sqrt{q}+1)/(4\sqrt{q}).
\end{eqnarray*}
Notice that this last condition is satisfied when $g\geq q$.

\bigskip

\begin{remark}
According to the results of Ihara \cite{Ihara}, the number of rational points of a (smooth, projective, absolutely irreducible) curve of genus $(g+1)$ over $\mathbb{F}_q$ is at most 
\begin{eqnarray}\label{Iha}
\frac{1}{2}\Bigl(2q-g+1+\sqrt{(8q+1)(g+1)^2+(4q^2-4q)(g+1)}\Bigr),
\end{eqnarray}
so using Proposition \ref{InequalityTau1}, we get another bound for $\vert\tau (P)\vert$. However, it is easy to check that the quantity (\ref{Iha}) is always (for any $q$ and $g$) greater than the second term of the inequality of Proposition \ref{InequalityTau3}.
\end{remark}

\bigskip

As in Corollary \ref{BornesN(X)}, we can derive some bounds on $\#  P({\mathbb F}_q)$ depending only on $g$ and $q$:

\begin{theorem}\label{Bornes}
If $g\geq q$, we have
$$(q+1-2 \sqrt{q})^g\leq m\left(-\psi\right)\leq \# P({\mathbb F}_q) \leq M\left(\psi\right)\leq(q+1+2 \sqrt{q})^g$$
where
$$\psi =\frac{g}{2g+1}\Bigl(q-g+\sqrt{(q-g)^2+(2g+1)(4gq+q^2+6q+1)}\Bigr).$$
\end{theorem}






\section{Prym varieties of dimension 2}

For any power of an odd prime $q$ and any integer $g\geq 1$, we define the quantities
$$Prym_q(g)=\max_{\pi}\# P_{\pi}(\mathbb{F}_q)\quad \mbox{ and }\quad prym_q(g)=\min_{\pi}\# P_{\pi}(\mathbb{F}_q)$$
where $\pi$ runs over the set of unramified double coverings of genus $(g+1)$ curves defined over $\mathbb{F}_q$.

\bigskip

Theorem \ref{Bornes} gives us bounds on $Prym_q(g)$ and $prym_q(g)$ when $g\geq q$. Here, we are interested in the case where $g$ is small compared to $q$. More precisely, the aim of this section is to determine $Prym_q(2)$ and $prym_q(2)$. To do so, we will exhibit maximal and minimal Prym surfaces. It turns out that it is enough to consider Prym varieties associated to coverings of hyperelliptic curves, and in this case, the \emph{Legendre construction} gives us an explicit description of the Prym variety. We start by recalling it; for more details, see \cite{mum2} and \cite{Bruin}.

\bigskip

Let $X$ be an hyperelliptic curve of genus $g$,  $p:X\longrightarrow {\mathbb P}^1$ be the associated double covering and $\{z_1,\ldots, z_{2g+2}\}$ be the set of branch points.
Then all unramified double coverings $\pi:Y\longrightarrow X$ arise as follows:

(1) Separate the branch points into two nonempty groups of even cardinality: $\{1,2,\ldots,2g+2\}=I_1\cup I_2$, $\#  I_1=2h+2$, $\#  I_2=2k+2$, $I_1\cap I_2=\emptyset$ (hence $h+k+1=g$).

(2) Consider the degree $2$ maps $p_1:X_1\longrightarrow {\mathbb P}^1$ and $p_2:X_2\longrightarrow {\mathbb P}^1$  with respective set of branch points $\{z_i\}_{i\in I_1}$ and $\{z_i\}_{i\in I_2}$.

(3) Let $Y$ be the normalization of $X\times_{{\mathbb P}^1}X_1$.

\bigskip
Then, we have such a diagram:
\bigskip

\begin{center}
\begin{tikzpicture}[scale=1.5]
\node (Y) at (0,1) {$Y$};
\node (X) at (-1,0) {$X$};
\node (X_1) at (0,0) {$X_1$};
\node (X_2) at (1,0) {$X_2$};
\node (P_1) at (0,-1) {${\mathbb P}^1$};
\path[->,font=\scriptsize,>=angle 90]
(Y) edge node[above]{$\pi$} (X)
(Y) edge node[above]{$\pi_2$} (X_2)
(Y) edge node[right]{$\pi_1$} (X_1)
(X_1) edge node[right]{$p_1$} (P_1)
(X_2) edge node[right]{$p_2$} (P_1)
(X) edge node[left]{$p$} (P_1);
\end{tikzpicture}
\end{center}


\bigskip

In this situation, the Prym variety $P_{\pi}$ associated to the covering $\pi:Y\longrightarrow X$ is isomorphic to the product of the Jacobians of $X_1$ and $X_2$:
$$P_{\pi}\simeq J_{X_1}\times J_{X_2}$$
(the isomorphism is given by $\pi_1^{\ast}+\pi_2^{\ast} : J_{X_1}\times J_{X_2}\longrightarrow P_{\pi}$, see \cite{Bruin}).

Moreover, if $I_1$ and $I_2$ are chosen to be stable under the action of $\Gal(\bar{\mathbb F}_q/{\mathbb F}_q)$ then all the curves and maps involved in this construction will be defined over ${\mathbb F}_q$.

\bigskip

In particular, we have:

\begin{proposition}\label{ProductJac}
Let $p_1:X_1\longrightarrow {\mathbb P}^1$ and $p_2:X_2\longrightarrow {\mathbb P}^1$ be degree $2$ maps with disjoints sets of ramified points. Then $J_{X_1}\times J_{X_2}$ is isomorphic to a Prym variety.
\end{proposition}

We deduce from Proposition \ref{ProductJac} some preliminary results describing when an abelian surface is a Prym variety:

\begin{proposition}\label{JacDim2}
A Jacobian of dimension 2 is isomorphic to a Prym variety.
\end{proposition}

\begin{proof}
A Jacobian of dimension 2 is the Jacobian of a (necessarily hyperelliptic) curve $C$ of genus 2.
Let $p:C\longrightarrow {\mathbb P}^1$ the associated double covering. Since $C$ has genus 2, $p$ is ramified at exactly 6 points. Since $\# {\mathbb P}^1({\mathbb F}_{q^2})=q^2+1\geq 3^2+1=10$, there exist unramified points $z_1, z_2\in{\mathbb P}^1({\mathbb F}_{q^2})$ such that the set $\{z_1, z_2\}$ is invariant under the action of $\Gal(\bar{\mathbb F}_q/\mathbb{F}_q)$ (since the whole set of ramified points is invariant under this action).
Now we consider the double covering $p_1:C_1\longrightarrow {\mathbb P}^1$ which is ramified at $z_1$ and $z_2$ (note that  $C_1$ is a genus zero curve) and we apply Proposition \ref{ProductJac}.
\end{proof}

\begin{proposition}\label{ProductElliptic}
If $E$ is an elliptic curve defined over $\mathbb{F}_q$ with a rational point $z_0$ of order strictly greater than 2 then the elliptic curve $\varphi_{z_0}(E)$ where
$$
\begin{matrix}
\varphi_{z_0}&:&E&\longrightarrow &E\cr
&& z &\longmapsto & z+z_0 \cr
\end{matrix}
$$
is isogenous to $E$, defined over $\mathbb{F}_q$ and has a set of ramified points disjoint from the one of $E$.

In particular, the product $E\times \varphi_{z_0}(E)$ is isomorphic to a Prym variety.
\end{proposition}

\begin{proof}
The translation by a rational point of order strictly greater than 2 sends the points of order 2 of $E$ on points of order strictly greater than 2. Hence Proposition \ref{ProductJac} gives the result.
\end{proof}

These two last propositions are sufficient to prove Theorem \ref{Prymq2} (giving the value of $Prym_q(2)$) and most cases of Theorem \ref{prymq2} (giving the value of $prym_q(2)$). In order to deal with the remaining cases of Theorem \ref{prymq2} (namely $q=3,5,9$), we will use the following result:

\begin{lemma}\label{Particular}
If $\# E({\mathbb F}_{q}) =1,2$ or $4$, then $E\times E$ is isogenous to a Prym variety.
\end{lemma}

\begin{proof}
Let $p:E\longrightarrow {\mathbb P}^1$ be a double covering defined over $\mathbb{F}_q$ and let $\{z_1,\ldots, z_{4}\}$ be the branch points. In the light of Proposition \ref{ProductJac}, it is enough to prove the existence of an automorphism $\varphi : {\mathbb P}^1\longrightarrow {\mathbb P}^1$ defined over $\mathbb{F}_q$ which sends $\{z_1,\ldots,z_{4}\}$ on a set disjoint with itself. In the remaining of the proof, we identify ${\mathbb P}^1(\bar{\mathbb F}_q)$ with $\bar{\mathbb F}_q\cup \{\infty\}$ in the usual way.

$\bullet$ Suppose that $\# E({\mathbb F}_{q}) =1$. Applying to ${\mathbb P}^1$ some suitable rational automorphism, we can assume that $z_1=0 $. The other branch points are of the form $z_i=\alpha_i$ where the set $\{\alpha_2,\alpha_3,\alpha_4\}$ is contained in ${\mathbb F}_{q^3}\setminus {\mathbb F}_{q}$ and invariant under the action of $\Gal(\mathbb F_{q^3}/\mathbb{F}_q)$. Consider the map $\varphi : x\mapsto 1/x$. If $\{i,j,k\}=\{2,3,4\}$, then on one hand, $\alpha_i\neq -1,1$, so $\alpha_i\neq 1/\alpha_i$ and on the other hand, we have $\alpha_i\neq 1/\alpha_j$ since otherwise the symmetric product $\alpha_1\alpha_2\alpha_3$ would be $\alpha_k$, which is not an element of $\mathbb{F}_q$. Therefore, the set of branch points and its image $\{\infty , 1/\alpha_2,1/\alpha_3,1/\alpha_4\}$ are disjoint and $\varphi$ satisfies the required conditions.

$\bullet$ Suppose that $\# E({\mathbb F}_{q}) =2$. If $q=3$, then by \cite{hnr}, $E\times E$ is isogenous to a Jacobian, so suppose that $q\geq 5$. The elliptic curve $E$ has $2$ rational branch points, so applying to ${\mathbb P}^1$ some suitable rational automorphism, we can assume that $z_1=0 $ and $z_2=1$. The other branch points are of the form $z_i=\alpha_i$ where the set $\{\alpha_3,\alpha_4\}$ is contained in ${\mathbb F}_{q^2}\setminus {\mathbb F}_{q}$ and invariant under the action of $\Gal(\mathbb F_{q^2}/\mathbb{F}_q)$. If $c\in {\mathbb F}_{q}$ then for $\{ i,j\}=\{ 3,4\}$, we have $\alpha_j\neq\alpha_i+c$, since otherwise the symmetric sum $\alpha_3+\alpha_4$ would be $2\alpha_i+c$, which is not an element of $\mathbb{F}_q$. Therefore, we can take $\varphi$ to be the translation by any element of ${\mathbb F}_{q}\setminus\{-1,0,1\}$ (which exists since $q\geq 5$).

$\bullet$ Suppose that $\# E({\mathbb F}_{q}) =4$. First, writing $\# E({\mathbb F}_{q}) =q+1+\tau $, we have $\vert \tau\vert =\vert 4-(q+1)\vert \leq 2\sqrt q$, which is possible if and only if $q\leq 9$. In \cite{ruck0},  R\"uck gives a list of the possible group structures for an elliptic curve. Applying his results, we find that if $q\leq 7$, then the group structure $\mathbb Z/4\mathbb Z$ is possible. Therefore, in these cases, we can choose an element in the isogeny class of $E$ which has a $4$-torsion point and apply Proposition \ref{ProductElliptic}.

Now suppose that $q=9$. According to \cite{ruck0}, the group $E({\mathbb F}_{9})$ must be isomorphic to $\mathbb Z/2\mathbb Z\times\mathbb Z/2\mathbb Z$. Therefore, $\{z_1,\ldots, z_{4}\}\subseteq {\mathbb P}^1(\mathbb F_9)$.
Applying to ${\mathbb P}^1$ some suitable rational automorphism, we can assume that $\{z_1, z_2,z_3\}= \{-1,0,1\}$. In the same way, we can also assume that $z_4\in \{\infty ,c\}$, where $c^2=-1$. Indeed, if $z_4\neq \infty$, then $z_4\in\mathbb F_9\setminus\mathbb F_3=\{\pm d,\pm d\pm 1\}$ with $d^2=-1$, therefore, possibly applying a translation by $\pm 1$ to ${\mathbb P}^1$ (notice that $\{z_1, z_2,z_3\}=\{-1,0,1\}$ is invariant by such a translation), we get what we want.

Consider the map $\varphi : x\mapsto (c+1)(x+c)/(x-c)$. We find that $\varphi (-1)=-(c-1)$, $\varphi (0)=-(c+1)$, $\varphi (1)=c-1$, $\varphi (\infty)=c+1$ and $\varphi (c)=\infty$. Therefore, $\{z_1, z_2,z_3,z_4\}$ and $\{\varphi (z_1),\varphi (z_2),\varphi (z_3),\varphi (z_4)\}$  are disjoint and $\varphi$ satisfies the required conditions.
\end{proof}

\bigskip

Remark that the proof of Proposition \ref{JacDim2} can be easily adapted to prove that Prym varieties of dimension 1 correspond to elliptic curves. Therefore, the values of $Prym_q(1)$ and $prym_q(1)$  can be directly derived from the Deuring-Waterhouse Theorem (see \cite{deur}, \cite{water}):  set $m=[2\sqrt{q}]$  and recall that $q=p^e$ with $p$ an odd prime number; we have:

\begin{proposition}\label{elliptic}
\noindent
\begin{enumerate}
\item $Prym_q(1)$ is equal to
\begin{itemize}
\item $(q+1+m)$ if $e=1$, $e$ is even or $p\not\vert m$
\item $(q+m)$ otherwise.
\end{itemize}
\item $prym_q(1)$ is equal to
\begin{itemize}
\item $(q+1-m)$ if $e=1$, $e$ is even or $p\not\vert m$
\item $(q+2-m)$ otherwise.
\end{itemize}
\end{enumerate}
\end{proposition}

\vskip1cm

Now, we focus on Prym surfaces. Les us first recall some basic facts about abelian surfaces. Let $A$ be an abelian surface over $ \mathbb{F}_q$ of type $[x_1,x_2]$.
Its characteristic polynomial has the form
$$f_A(t)=t^4+a_1t^3+a_2t^2+qa_1t+q^2,$$
with
$$ a_1=x_1+x_2 \quad \text{and} \quad a_2=x_1x_2+2q.$$
By elementary computations, R\"uck \cite{ruck} showed that  the fact that the roots of $f_A(t)$ are $q$-Weil numbers (i.e. algebraic integers such that their images under every complex embedding have absolute value $q^{1/2}$) is equivalent to
\begin{equation}
\vert {a_1}\vert \leq 2m \quad \text{and} \quad 2 \vert{a_1}\vert q^{1/2}-2q\leq a_2\leq\frac{a_1^2}{4}+2q \label{a12}
\end{equation}
where $m=[2\sqrt{q}]$. We have
\begin{equation}
\# {A(\mathbb{F}_q)} =f_A(1)=q^2+1+(q+1)a_1+a_2.\label{abeliansurface}
\end{equation}

As described in \cite{AHLacta},
Table  \ref{tabmax} gives all the possibilities for $(a_1,a_2)$ such that $a_1\geq 2m-2$. The numbers of points are classified in decreasing order and an abelian variety with $(a_1,a_2)$ not in the table has a number of points strictly less than the values of the table. Here
$$\varphi_{1} = (- 1 + \sqrt{5})/2, \quadÊ\varphi_{2} = (- 1 - \sqrt{5})/2.$$

\begin{table}[htbp]
\begin{center}
\renewcommand{\arraystretch}{1.5}
\begin{tabular}{| l l | l | l |}
\hline
$a_1$ & $a_2$ & Type & $\# {A(\mathbb{F}_q)} $\\
\hline
\hline
$2m$ & $m^2+2q$ & $[m,m]$ & $b^2$ \\
\hline
$2m-1$ & $m^2-m+2q$ & $[m,m-1]$ & $b(b - 1)$ \\
 & $m^2-m-1+2q$ & $[m + \varphi_{1},m + \varphi_{2}]$ & $b^2 - b - 1$ \\
\hline
$2m-2$ & $m^2-2m+1+2q$ & $[m-1,m-1]$ & $(b - 1)^2$ \\
 & $m^2-2m+2q$ & $[m,m-2]$ & $b(b - 2)$ \\
 & $m^2-2m-1+2q$ & $[m-1+\sqrt{2},m-1-\sqrt{2}]$ & $(b - 1)^2 - 2$ \\
 & $m^2-2m-2+2q$ & $[m-1+\sqrt{3},m-1-\sqrt{3}]$ & $(b - 1)^2 - 3$ \\
\hline
\end{tabular}
\end{center}
\caption{Couples $(a_1,a_2)$ maximizing $\# {A(\mathbb{F}_q)} $ $(b = q + 1 + m)$.}
\label{tabmax}
\end{table}

In the same way, we build the table of couples $(a_1,a_2)$ with $a_1 \leq - 2 m + 2$. If $q>5$, the numbers of points are classified in increasing order and an abelian variety with $(a_1,a_2)$ not in the following table has a number of points strictly greater than the values of the table (see \cite{AHLacta}).

\begin{small}
\begin{table}[htbp]
\begin{center}
\renewcommand{\arraystretch}{1.5}
\begin{tabular}{| l l | l | l |}
\hline
$a_1$ & $a_2$ & Type & $A(\mathbb{F}_q)$ \\
\hline
\hline
$-2m$ & $m^2+2q$ & $[-m,-m]$ & $b'^2$ \\
\hline
$-2m+1$ & $m^2-m-1+2q$ & $[-m-\varphi_{1},-m-\varphi_{2}]$ & $b'^2 + b' - 1$ \\
 & $m^2-m+2q$ & $[-m,-m+1]$ & $b'(b' + 1)$ \\
\hline
$-2m+2$ & $m^2-2m-2+2q$ & $[-m+1+\sqrt{3},-m+1-\sqrt{3}]$ & $(b' + 1)^2 - 3$ \\
 & $m^2-2m-1+2q$ & $[-m+1+\sqrt{2},-m+1-\sqrt{2}]$ & $(b' + 1)^2 - 2$ \\
 & $m^2-2m+2q$ & $[-m,-m+2]$ & $b'(b' + 2)$ \\
 & $m^2-2m+1+2q$ & $[-m+1,-m+1]$ & $(b' + 1)^2$ \\
\hline
\end{tabular}
\end{center}
\caption {Couples $(a_1,a_2)$ minimizing $\# {A(\mathbb{F}_q)} $ $(q>5$ and $b' = q + 1 - m)$.}
\label{tabmin}
\end{table}
\end{small}

\vskip1cm

Theorem \ref{Prymq2} and \ref{prymq2} will be proved in the following way:

\begin{enumerate}
\item Look at the highest row of Table \ref{tabmax} or \ref{tabmin} (depending on the theorem being proved).
\item Check if the corresponding polynomial is the characteristic polynomial of an abelian variety.
\item When it is the case, check if this abelian variety is isogenous to a Prym  variety.
\item When it is not the case, look at the next row and come back to the second step.
\end{enumerate}

For the second step, we use the results of  R\"uck \cite{ruck} (completed by Maisner, Nart and Xing) describing set of characteristic polynomials of abelian surfaces. More precisely, we use two facts: first, a simple abelian surface with a reducible characteristic polynomial must have trace $0$ or $\pm 2\sqrt q$, and thus, excluding these two cases if $x_1,x_2$ are integers, there exists an abelian surface of type $[x_1,x_2]$ if and only if there exists two elliptic curves of respective trace $x_1$ and $x_2$ (and in this case, the corresponding isogeny class contains the product of these elliptic curves). Secondly, if $(a_1,a_2)$ satisfy (\ref{a12}) and $p$ does not divide $a_2$ then the corresponding polynomial is the characteristic polynomial of an abelian surface.

For the third step,  we use Proposition \ref{JacDim2} combined with results from \cite{hnr} (which gives a description of the set of isogeny classes of abelian surfaces containing a Jacobian), Proposition \ref{ProductElliptic} and Lemma \ref{Particular}.

Note that this method is still valid for $q\leq 5$, even if Table \ref{tabmin} is not correct anymore. Indeed, we always have $prym_q(2)\geq (q+1-m)^2$ and for $q\leq 5$ there exists a Prym surface with $(q+1-m)^2$ points (the conditions of the first point of Theorem \ref{prymq2} are satisfied for $q=3,5$).

Finally, remark that we need to have $\vert x_i\vert\leq 2\sqrt q$, $i=1,2$, thus

\begin{itemize}
\item[$\bullet$] in order to have the existence of an abelian surface with $(x_1,x_2)=\pm (m+\frac{-1+\sqrt{5}}{2},m+\frac{-1-\sqrt{5}}{2})$, it is necessary that $m+\frac{-1+\sqrt{5}}{2}\leq 2\sqrt q$, which is equivalent to $\{ 2\sqrt q\} =2\sqrt q -m\geq\frac{\sqrt 5 -1}{2}\simeq 0,61803$ (where $\{x\}$ denotes the fractional part of $x$ i.e. $\{x\}=x-[x]$),

\item[$\bullet$] in order to have the existence of an abelian surface with $(x_1,x_2)=\pm (-m+1+\sqrt{2},-m+1-\sqrt{2})$, it is necessary that $\{ 2\sqrt q\} \geq\sqrt 2 -1\simeq 0,41421$,

\item[$\bullet$] in order to have the existence of an abelian surface with $(x_1,x_2)=\pm (-m+1+\sqrt{3},-m+1-\sqrt{3})$, it is necessary that  $\{ 2\sqrt q\} \geq\sqrt 3 -1\simeq 0,73205$.
\end{itemize}

\bigskip

\begin{theorem} \label{Prymq2}
If $q=p^e$, then $Prym_q(2)$ is equal to
\begin{itemize}
\item[$\bullet$] $(q+1+m)^2$ if $e=1$ or $e$ even or $p\not\vert m$
\item[$\bullet$] $(q+1+m-\frac{1+\sqrt{5}}{2})(q+1+m-\frac{1-\sqrt{5}}{2})$ if $e\neq 1$, $e$ odd, $p\vert m$ and $\{ 2\sqrt q\}\geq\frac{\sqrt 5 -1}{2}$
\item[$\bullet$] $(q+m)^2$ else.
\end{itemize}
\end{theorem}

\begin{proof}
$\bullet$ There exists an abelian variety of type $[m,m]$ if and only if $e=1$ or $e$ is even or $p\not\vert m$. If it is the case, the corresponding isogeny class contains the product of elliptic curves of trace $-m$ and these curves have $q+1+m\geq 3+1+3=7$ rational points, thus at least one rational point of order $>2$ (the group of $2$-torsion points of an elliptic curve is isomorphic to ${\mathbb Z}/2{\mathbb Z}\times{\mathbb Z}/2{\mathbb Z}$). Then, we apply Proposition \ref{ProductElliptic} to conclude that there exists a Prym variety with $(q+1+m)^2$ rational points in this case.

$\bullet$ Else, there does not exist an abelian variety of type $[m,m-1]$ and there exists an abelian variety of type $[m+\frac{-1+\sqrt{5}}{2},m+\frac{-1-\sqrt{5}}{2}]$ if and only if $\{ 2\sqrt q\}\geq\frac{\sqrt 5 -1}{2}$ (since $p\vert m$ thus $p\not\vert a_2=m^2-m-1+2q$). If it is the case, the corresponding isogeny class contains a Jacobian of dimension 2 (see \cite{hnr}) which is isomorphic to a Prym variety by Proposition \ref{JacDim2}.

$\bullet$ Else, using the same arguments as in the first point, we deduce that the product of elliptic curves of trace $-(m-1)$ (such curves exist since $p\vert m$ hence $p\not\vert (m-1)$) is isogenous to a Prym variety.

\end{proof}

\begin{theorem} \label{prymq2}
If $q=p^e$, then $prym_q(2)$ is equal to
\begin{itemize}
\item[$\bullet$] $(q+1-m)^2$ if $e=1$ or $e$ even or $p\not\vert m$
\item[$\bullet$] $(q+1-m+\frac{1+\sqrt{5}}{2})(q+1-m+\frac{1-\sqrt{5}}{2})$ if $e\neq 1$, $e$ odd, $p\vert m$ and $\{ 2\sqrt q\}\geq\frac{\sqrt 5 -1}{2}$
\item[$\bullet$] $(q+2-m-\sqrt{2})(q+2-m+\sqrt{2})$ if $e\neq 1$, $e$ odd, $p\vert m$ and $\sqrt{2}-1\leq\{ 2\sqrt q\}<\frac{\sqrt 5 -1}{2}$
\item[$\bullet$] $(q+2-m)^2$ else.
\end{itemize}
\end{theorem}

\begin{proof}
$\bullet$ There exists an abelian variety of type $[-m,-m]$ if and only if $e=1$ or $e$ is even or $p\not\vert m$. If it is the case, the corresponding isogeny class contains the product of elliptic curves of trace $m$. Such elliptic curves have $q+1-m$ rational points, which is greater than or equal to $11+1-6=6$ if $q\geq 11$ and equal to $7+1-5=3$ if $q=7$, thus, in these cases, they must have a rational point of order $>2$ and Proposition \ref{ProductElliptic} applies. For $q=3,5,9$, we apply Lemma \ref{Particular}.

$\bullet$ Else, there does not exist an abelian variety of type
$[-m,-(m-1)]$ and there exists an abelian variety of type
$[-m+\frac{1+\sqrt{5}}{2},-m+\frac{1-\sqrt{5}}{2}]$ if and only if $\{ 2\sqrt q\}\geq\frac{\sqrt 5 -1}{2}$ (since $p\vert m$ thus $p\not\vert a_2=m^2-m-1+2q$). If it is the case,  the corresponding isogeny class contains a Jacobian of dimension 2 (see \cite{hnr}) which is isomorphic to a Prym variety by Proposition \ref{JacDim2}.

$\bullet$ Else, there does not exist an abelian variety of type $[-m+1+\sqrt{3},-m+1-\sqrt{3}]$ (since  $\{ 2\sqrt q\}<\frac{\sqrt 5 -1}{2}<\sqrt 3 -1$)  and there exists an abelian variety of type
$[-m+1+\sqrt{2},-m+1-\sqrt{2}]$ if and only if  $\{ 2\sqrt q\} \geq\sqrt 2 -1$ (since $p\vert m$ hence $p\not\vert a_2=m^2-2m-1+2q$). If it is the case, once again,  the corresponding isogeny class contains a Jacobian of dimension 2 (see \cite{hnr}) which is isomorphic to a Prym variety by Proposition \ref{JacDim2}.

$\bullet$  Else, there does not exist an abelian variety of type $[-m,-(m-2)]$ and the product of elliptic curves of trace $-(m-1)$ (such curves exist since $p\vert m$ hence $p\not\vert (m-1)$) is isogenous  to a Prym variety as in the first point.
\end{proof}

\vskip1cm

\begin{remark}
We can define $N_k(P)=q^k+1+\tau_k(P)$, these are the "virtual numbers of rational points" of $P$. If $q\leq 9$, then $q+1-2m=-2$ and Theorem \ref{prymq2} asserts that there exists Prym surfaces of type $[-m,-m]$. This gives us examples of Prym varieties with $N_1(P)<0$. In particular, the bounds announced in \cite{AHLcras} and proved in \cite{AHLacta} on the number of rational points on abelian varieties with nonnegative virtual numbers of rational points do not apply.
\end{remark}







\begin{thebibliography}{99}


\bibitem{AHLcras}
Y. Aubry, S. Haloui, G. Lachaud. Sur le nombre de points rationnels des vari\'et\'es ab\'eliennes et des Jacobiennes sur les corps finis.
\emph{C. R. Acad. Sci. Paris}, Ser. I 350 (2012) 907-910.


 \bibitem{AHLacta}
Y. Aubry, S. Haloui, G. Lachaud. On the number of points on abelian and Jacobian varieties over finite fields, to appear in Acta Arithmetica (2013).



\bibitem{Beau}
A. Beauville.
\emph{Prym varieties: a survey}.
Proc. symposia in Pure Math. 49 (1989).


\bibitem{Bruin}
N. Bruin. The arithmetic of Prym varieties in genus 3.
\emph{Composition Mathematica}.
 Vol. 144, p. 317-338, 2008.

\bibitem{deur}
M. Deuring. Die typen der multiplikatorenringe elliptischer funktionenkörper.
\emph{Abh. Math. Sem. Hansischen Univ.}
{\bf 14} (1941), 197-272.

\bibitem{hnr}
E. Howe, E. Nart, C. Ritzenthaler. Jacobians in isogeny classes of abelian surfaces over finite fields.
\emph{Ann. Inst. Fourrier, Grenoble}.
 no 59, p. 239-289, 2009.


\bibitem{Ihara}
Y. Ihara, Some remarks on the number of rational points of algebraic curves over finite fields.
\emph{J. Fac. Sci. Univ. Tokyo Sect. IA Math.} {\bf 28} (1981), no. 3, 721–724 (1982).




\bibitem{mana}
D. Maisner, E. Nart, appendice de E. W. Howe. Abelian surfaces over finite fields as jacobians.
\emph{Experiment. Math.}.
Vol 11, p. 321-337, 2002.




\bibitem{mum2}
D. Mumford. Prym varieties, I
\emph{Contributions to Analysis}.
(Academic Press, 1974), 325-350.

\bibitem{per}
M. Perret. Number of points of Prym varieties over finite fields.
\emph{Glasgow Math. J.}.
Vol 48, p. 275-280, 2006.

\bibitem{ruck0}
H. G. R\"uck. A note on elliptic curves over finite fields.
\emph{Math. Comp.},
Vol 49, p. 301-304, 1987.

\bibitem{ruck}
H. G. R\"uck. Abelian surfaces and Jacobian varieties over finite fields.
\emph{Compositio Math.},
Vol 76, p. 351-366, 1990.







\bibitem{water}
W.C. Waterhouse. Abelian varieties over finite fields.
\emph{Ann. Sc. E.N.S.},
(4), 2, 1969, 521-560.


\end{thebibliography}
\end{document}